\documentclass[leqno,11pt]{amsart}
\usepackage{}
\usepackage{amssymb, amsmath}
\usepackage{amsthm, amsfonts,mathrsfs}
\def\inte#1{
\displaystyle\mathop{#1\kern0pt}^\circ }

\def\cF{{\mathcal F}}

\def\grad{\nabla}

\def\virgp{\raise 2pt\hbox{,}}
\def\cdotpv{\raise 2pt\hbox{;}}

\def\C{\mathop{\mathbb C\kern 0pt}\nolimits}
\def\DD{\mathop{\mathbb D\kern 0pt}\nolimits}
\def\EE{\mathop{{\mathbb E \kern 0pt}}\nolimits}
\def\K{\mathop{\mathbb K\kern 0pt}\nolimits}
\def\N{\mathop{\mathbb N\kern 0pt}\nolimits}
\def\Q{\mathop{\mathbb Q\kern 0pt}\nolimits}
\def\R{\mathop{\mathbb R\kern 0pt}\nolimits}
\def\SS{\mathop{\mathbb S\kern 0pt}\nolimits}
\def\ZZ{\mathop{\mathbb Z\kern 0pt}\nolimits}
\def\TT{\mathop{\mathbb T\kern 0pt}\nolimits}
\def\P{\mathop{\mathbb P\kern 0pt}\nolimits}

\newcommand{\beq}{\begin{equation}}
\newcommand{\eeq}{\end{equation}}
\newcommand{\ben}{\begin{eqnarray}}
\newcommand{\een}{\end{eqnarray}}
\newcommand{\beno}{\begin{eqnarray*}}
\newcommand{\eeno}{\end{eqnarray*}}

\newtheorem{thm}{Theorem}[section]
\newtheorem{lem}{Lemma}[section]

\renewcommand{\theequation}{\thesection.\arabic{equation}}
\begin{document}
\title[well-posedness of a kind of the free surface equation of shallow water wave]
{well-posedness of a kind of the free surface equation of shallow water wave}

\author[M.M.D]{MiaoMiao Dang}
\address[]{School of Mathematics, Northwest University, Xi'an 710127, China}
\email{dmm0819@163.com}
\author[Z.Li]{Zhouyu Li}
\address[]{School of Mathematics, Northwest University, Xi'an 710069, China}
\email{zylimath@163.com}

\maketitle
\begin{abstract}
This paper is concerned with the Cauchy problem of the one-dimensional free surface equation of shallow water wave,
we obtain local well-posedness of the free surface equation of shallow water wave in Sobolev spaces. In addition, we
also derive a wave-breaking mechanism for strong solutions.


\end{abstract}

\vskip 0.2cm


\renewcommand{\theequation}{\thesection.\arabic{equation}}
\setcounter{equation}{0}

\section{Introduction}

For one-dimensional surfaces, the water waves equations read in the following nondimensionalized form
\begin{equation}\label{CH-0}
    \begin{cases}
        \mu\partial^{2}_x\Phi+\partial_z\Phi^{2}=0 \; &  \text{in} \; \Omega_{t},\\
        \partial_z\Phi=0 \quad &\text{at} \; z=-1,\\
        \partial_t\eta-\displaystyle\frac{1}{\mu}(-\mu\partial_x\eta\partial_x\Phi+\partial_z\Phi)=0 \quad & \text{at}\; z=\varepsilon\eta,\\
        \partial_t\Phi+\displaystyle\frac{\varepsilon}{2}(\partial_{x}\Phi)^{2}+\displaystyle\frac{\varepsilon}{2\mu}(\partial_{z}\Phi)^{2}=0 \quad & \text{at} \; z=\varepsilon\eta,
    \end{cases}
    \end{equation}
where $\varepsilon$ and $\mu$ are two dimensionless parameters defined as
$$\varepsilon=\displaystyle\frac{a}{h},\quad \mu=\displaystyle\frac{h^{2}}{\lambda^{2}},$$
and $h$ is the mean depth, $a$ is the typical amplitude and $\lambda$ the typical wavelength
of the waves under consideration.
Where $x\longmapsto \varepsilon\eta(t, x)$ parameterizes the elevation of the free surface at time $t$,
 $\Omega_{t}=\{(x, z), -1<z<\varepsilon\eta(t, x)\}$ is the fluid domain delimited by the free surface
and the flat bottom $\{z=-1\}$, and where $\Phi(t, \cdot)$ (defined on $\Omega_{t}$) is the velocity
potential associated to the flow (that is, the two-dimensional velocity field $v$ is given
by $v=(\partial_{x}\Phi,\partial_t\Phi)^{T})$.

Making assumptions on the respective size of $\varepsilon$
and $\mu$, one is led to derive (simpler) asymptotic models from \eqref{CH-0}.
In the shallow-water scaling ($\mu\ll1$), one can derive the so-called Green-Naghdi (GN) equations (see \cite{Green1976} for the derivation, and \cite{Samaniego1} for a rigorous justification),
without any smallness assumption on $\varepsilon$ (that is, $\varepsilon= O(1)$).
\begin{equation}\label{CH-1}
    \begin{cases}
      \eta_{t}+[(1+\varepsilon\eta)u]_{x}=0, \\
       u_{t}+\eta_{x}+\varepsilon uu_{x}=\displaystyle\frac{\mu}{3}\displaystyle\frac{1}{1+\varepsilon\eta}[{1+\varepsilon\eta}^{3}(u_{xt}+\varepsilon uu_{x}-\varepsilon u^{2}_{x})]_{x},
    \end{cases}
    \end{equation}
where $u(t, x)=\displaystyle\frac{1}{1+\varepsilon\eta}\int^{\varepsilon\eta}_{-1}\partial_{x}\Phi(t, x, z)dz$ denotes vertically
averaged horizontal component of the velocity.

Because of the complexity of water-waves problem, they are often replaced for practical purposes by
approximate asymptotic systems. The most prominent examples are the GN equations, which is a widely
used model in coastal oceanography.

A recent rigorous justification of the GN model was given by \cite{Li2006} in 1D and for flat bottoms, which is  based on the energy estimates and the proof of the well-posedness for the GN equations and the water wave problem , and by  Alvarez-Samaniego and Lannes \cite{Samaniego2} allowed losses of derivatives in this energy estimate and therefore construct a solution by a Nash-Moser iterative scheme and proved the well-posedness of the GN equation in 1D and 2D and discuss the problem of their validity as asymptotic models for the water-waves equations.

Recently, Gui and Liu  \cite{Gui-Liu-Sun2016} consider the 1-D R-CH equation is well-posed and show that the deviation of the free surface can be determined by the horizontal velocity at a certain depth in the second-order approximation.

In this paper, we consider the well-posedness of the free surface equation, which approximate solutions consistent with the GN equation.

The family of equations
\begin{equation}\label{CH-2}
\eta_{t}+\eta_{x}+\displaystyle\frac{3}{2}\varepsilon\eta\eta_{x}-\displaystyle\frac{3}{8}\varepsilon^{2}\eta^{2}\eta_{x}+\displaystyle\frac{3}{16}\varepsilon^{3}\eta^{3}\eta_{x}
+\mu(\alpha\eta_{xxx}+\beta\eta_{xxt})=\varepsilon\mu(\gamma\eta\eta_{xxx}+\delta\eta_{x}\eta_{xx}),
\end{equation}
for the evolution of the surface elevation can be used to construct an approximate
solution consistent with the GN equations(see \cite{Constantic2009}). Where $\varepsilon, \mu, \alpha, \beta, \gamma$ and $\delta$ are constants.

Let $q\in \R$ and assume that
$$\alpha=q,\beta=q-\displaystyle\frac{1}{6},\gamma=-\displaystyle\frac{3}{2}q-\displaystyle\frac{1}{6},\delta=-\displaystyle\frac{9}{2}q-\displaystyle\frac{5}{24}$$
especially, choosing $q=\displaystyle\frac{1}{12},\mu=12,\varepsilon=1$ the equation \eqref{CH-2} reads
\begin{equation}\label{CH-3}
\eta_{t}+\eta_{x}+\displaystyle\frac{3}{2}\eta\eta_{x}-\displaystyle\frac{3}{8}\eta^{2}\eta_{x}+\displaystyle\frac{3}{16}\eta^{3}\eta_{x}
+\eta_{xxx}-\eta_{xxt}=-\displaystyle\frac{7}{2}\eta\eta_{xxx}-7\eta_{x}\eta_{xx}.
\end{equation}
\quad In this paper, we will investigate well-posedness of the Cauchy problem of the equivalent form of the free surface equation \eqref{CH-3}:
\begin{equation}\label{1.4}
    \begin{cases}
   & \partial_{t}\eta-\partial_{x}\eta-\frac{7}{2}\eta\partial_{x}\eta=(1-\partial_{x}^{2})^{-1}\partial_{x}
       \big(-2\eta-\frac{5}{2}\eta^{2}+\frac{7}{4}(\partial_{x}\eta)^{2}+\frac{1}{8}\eta^{3}-\frac{3}{64}\eta^{4}\big),\\
       &\qquad\qquad\qquad\qquad\qquad\qquad\qquad\qquad\qquad\qquad\qquad\qquad \forall \, t > 0,\, x \in \R,\\
       &\eta|_{t=0}=\eta_{0},\qquad \qquad \forall \, x\in \R,
    \end{cases}
\end{equation}
and show the wave-breaking phenomenon. Our main results are stated as follows.

\begin{thm}\label{thm-main-1}
Suppose that $\eta_{0}\in H^{s}(\R)$ with $s>\frac{3}{2},$ then there exist a positive time $T>0$ such that, the equation \eqref{1.4} has a unique strong solution $\eta\in{\mathcal{C}}([0,T];H^{s})\cap\mathcal{C}^{1}([0,T]; H^{s-1})$ and the map $\eta_{0}\mapsto \eta$ is continuous from a neighborhood of $\eta_{0}$ in $H^{s}$ into $\eta \in {\mathcal{C}}([0,T];H^s)\cap{\mathcal{C}}^{1}([0,T]; H^{s-1}).$ Moreover, the energy
\begin{equation}\label{lem-1}
H=H(\eta)=\frac{1}{2}\int_{\R}(\eta^{2}+\eta_{x}^{2})dx,
\end{equation}
is independent of the existence time $t\in[0,T).$
\end{thm}

\begin{thm}\label{thm-main-2}
 Let $\eta_{0}\in H^{s}$ be as in Theorem \ref{thm-main-1} with $s>\frac{3}{2}.$ Let $\eta$ be the corresponding solution to \eqref{1.4}. Assume $T_{\eta_{0}}^{\ast}>0$ is the maximal existence time. Then
\begin{equation}\label{CH-y1}
T_{\eta_{0}}^{\ast}<\infty\Rightarrow\int_{0}^{T_{\eta_{0}}^{\ast}}\|\partial_{x}\eta(\tau)\|_{L^{\infty}}d\tau=\infty.
\end{equation}
\end{thm}
\medbreak \noindent{\bf Remark:} The blow-up criterion \eqref{CH-y1} implies that the lifespan $T_{\eta_{0}}^{\ast}$ does not depend on the regularity index $s$ of the initial data $\eta_{0}.$
\begin{thm}\label{thm-main-3}
 Let $\eta_{0}\in H^{s}(\mathbb{R})$ with $s>\frac{3}{2}$, $C_{0}=\frac{1}{2}+3\|\eta_{0}\|_{H^{1}}^{2}+\frac{3}{16}\|\eta_{0}\|_{H^{1}}^{3}+\frac{3}{32}\|\eta_{0}\|_{H^{1}}^{4}$. Assume that the initial value $\eta_{0}$ satisfies that $\eta_{0x}(x_{0})>\sqrt{\frac{2}{7}C_{0}}$ with the point $x_{0}$ defined by $\eta_{0x}(x_{0})=\inf\limits_{x\in\R}\eta_{0x}(x)$. Then the corresponding solution to the \eqref{1.4} blows up in finite time in the following sense: there exists a $T_{0}$ with $0<T_{0}\leq\displaystyle\frac{2}{7(1-\sigma)\eta_{0x}(x_{0})}$ such that
 \begin{equation}\label{CH-4}
  \begin{split}
\limsup_{t\rightarrow T_{0}}(\inf_{x\in\R}\eta_{x}(t,x))=+\infty,
 \end{split}
\end{equation}
where $\sigma\in(0, 1)$ such that $\sqrt{\sigma}\eta_{0x}(x_{0})=\sqrt{\frac{2}{7}C_{0}}.$
\end{thm}

This paper is organized as follows. In Section 2, we collect some elementary
facts and inequalities which will be used later. Section 3 is devoted to the local well-posedness of the free surface system
\eqref{1.4}. Finally, using the transport equation theory, we can give the wave-breaking phenomenon Theorem \ref{thm-main-3} in Section 4.

Let us complete this section with the notations we are going to use
in this context.

\medbreak \noindent{\bf Notations:} Let $A, B$ be two operators, we
denote $[A,B]=AB-BA,$ the commutator between $A$ and $B$. We shall
denote by $(a, b)$ (or $(a,  b)_{L^2}$) the $L^2(\R)$ inner product
of $a$ and $b$, and $\int\cdot dx\triangleq\int_{\mathbb{R}}\cdot dx$.

 We always denote the Fourier transform of a function $u$ by
$\hat{u}$ or $\cF(u)$. For $s \in \mathbb{R}$, we denote the pseudo-differential operator $\Lambda^s:=(1-\Delta)^{\frac{s}{2}}$ with the Fourier symbol $(1+|\xi|^2)^{\frac{s}{2}}$. Note that if $g(x)=\frac{1}{2}e^{-|x|},$ $x\in\R,$ then $(1-\partial_{x}^{2})^{-1}f=g\ast f$ for all $f\in L^{2}(\R),$ where $\ast$ denotes the spatial convolution. To simplify the notations, we shall use the letter $C$ to
denote a generic constant which may vary from line to line.

For $X$ a Banach space and $I$ an interval of $\R,$ we denote by
${\mathcal{C}}(I;\,X)$ the set of continuous functions on $I$ with
values in $X,$ for $q\in[1,+\infty],$ the
notation $L^q(I;\,X)$ stands for the set of measurable functions on
$I$ with values in $X,$ such that $t\longmapsto\|f(t)\|_{X}$ belongs
to $L^q(I).$


\renewcommand{\theequation}{\thesection.\arabic{equation}}
\setcounter{equation}{0} 

\section{Preliminaries}

In this section, we will give some elementary facts and useful lemmas which will
be used in the next section.

Let us first recall some basic facts about the regularizing operator called a mollifier,
see \cite{Majda2002} for more details. Given any radial function
\begin{eqnarray*}
\rho(|x|)\in \mathcal{C}_0^\infty(\mathbb{R}^N), \quad \rho\geq0, \quad\int_{\mathbb{R}^N}\rho dx=1,
\end{eqnarray*}
define the mollification $\mathcal{J}_\varepsilon u$ of $u\in L^p(\mathbb{R}^N)$, $1\leq p\leq \infty$,
by
\begin{eqnarray}\label{mollifier-1}
(\mathcal{J}_\varepsilon u)(x)=\varepsilon^{-N}\int_{\mathbb{R}^N}\rho(\displaystyle\frac{x-y}{\varepsilon})u(y) dy,
\quad \varepsilon>0.
\end{eqnarray}

Mollifiers have several well-known properties: \\
(i). $\mathcal{J}_\varepsilon u$ is a $\mathcal{C}^\infty$ function;\\
(ii). for all $u\in \mathcal{C}^0(\mathbb{R}^N)$, $\mathcal{J}_\varepsilon u\rightarrow u$ uniformly on any compact
set $\Omega$ in $\mathbb{R}^N$ and $\|\mathcal{J}_\varepsilon u\|_{L^\infty}\leq\|u\|_{L^\infty}$;\\
(iii). mollifiers commute with distribution derivatives, $D^{\alpha}\mathcal{J}_\varepsilon u=\mathcal{J}_\varepsilon D^\alpha u$;\\
(iv). for all $u\in H^m(\mathbb{R}^N)$, $\mathcal{J}_\varepsilon u$ converges to $u$ in $H^m$
and the rate of convergence in the $H^{m-1}$ norm is linear in $\varepsilon$:
$\lim_{\varepsilon\rightarrow0}\|\mathcal{J}_\varepsilon u-u\|_{H^m}=0, \, \|\mathcal{J}_\varepsilon u-u\|_{H^{m-1}}\leq C\varepsilon\|u\|_{H^m}$;\\
(v). for all $u\in H^m(\mathbb{R}^N)$, $k\in\mathbb{Z}^+\cup \{0\}$, and $\varepsilon>0$,
$
\|\mathcal{J}_\varepsilon u\|_{H^{m+k}}\leq\frac{C(m, k)}{\varepsilon^k}\|u\|_{H^m}, \, \|\mathcal{J}_\varepsilon u\|_{L^\infty}\leq\frac{C(k)}{\varepsilon^{\frac{N}{2}+k}}\|u\|_{L^2}.
$
\begin{lem}[Aubin-Lions's lemma, \cite{Simon1990}]\label{lem-Lions-Aubin's}
Assume $X\subset E\subset Y$ are Banach spaces and $X\hookrightarrow\hookrightarrow E$.
Then the following embeddings are compact:

(i)$\left\{\varphi:\varphi\in L^q([0, T]; X), \displaystyle\frac{\partial \varphi}{\partial t}\in L^1([0, T]; Y)\right\}\hookrightarrow\hookrightarrow
 L^q([0, T]; E)\quad if \quad 1\leq q\leq\infty$;

(ii)$\left\{\varphi:\varphi\in L^\infty([0, T]; X), \displaystyle\frac{\partial \varphi}{\partial t}\in L^r([0, T]; Y)\right\}\hookrightarrow\hookrightarrow
 {\mathcal{C}}([0, T]; E)\quad if \quad 1\leq r\leq\infty$.
\end{lem}

\begin{lem}[Calculus inequalities, \cite{Klainerman1981}]\label{lem-Calculus inequalities}
Let $s>0$. Then the following two estimates are true:

(i) $\|uv\|_{H^s(\R)}\leq C\|u\|_{H^s(\R)}\|v\|_{H^s(\R)} \quad \mbox{for all} \quad s>\frac{1}{2}$;\\

(ii) $\|[\Lambda^s, u]v\|_{L^2(\R)} \leq C(\|u\|_{H^s}\|v\|_{L^\infty(\R)}+\|\grad u\|_{L^\infty(\R)}\|v\|_{H^{s-1}(\R)})$,

\noindent where all the constants $C$s are independent of $u$ and $v$.
\end{lem}

\begin{lem}[1-D Moser-type estimates, \cite{Chemin2004}]\label{lem-1-D Moser-type estimates}
The following estimates holds:\\
(i)For $s\geq0$,
$$\|fg\|_{H^s(\R)}\leq C\{\|f\|_{H^s(\R)}\|g\|_{L^\infty(\R)}+\|g\|_{H^s(\R)}\|f\|_{L^\infty(\R)}\}.$$
(ii))For $s_{1}\leq\frac{1}{2},s_{2}>\frac{1}{2}$ and  $s_{1}+s_{2}>0$,
$$\|fg\|_{H^{s_{1}}(\R)}\leq C\|f\|_{H^{s_{1}}(\R)}\|g\|_{H^{s_{2}}(\R)}.$$
\end{lem}

To study the wave-breaking criterion of the system \eqref{1.4}, we need the following lemma
on the transport equation (especially taking the space dimension $d = 1$).
\begin{lem}[Transport equation theory, \cite{Chemin2001,Gui-Liu2010}]\label{lem-Transport equation theory}
Suppose that $s>-\frac{d}{2}$. Let $\upsilon$ be a vector field such that $\nabla\upsilon$ belongs to $L^{1}([0,T];H^{s-1})$ if $s>1+\frac{d}{2}$ or to $L^{1}([0,T];H^{\frac{d}{2}}\bigcap L^{\infty})$otherwise. Suppose also that $f_{0}\in H^{s},F\in L^{1}([0,T];H^{s})$ and that $f\in L^{\infty}([0,T];H^{s})\bigcap \mathcal{C}([0,T];S^{'})$solves the d-dimensional linear transport equations
\begin{equation}\label{lem-2.4}
    \begin{cases}
       \partial_{t}f+\upsilon \cdot \nabla f= F,\, \forall \, t>0, \, x \in \R^d,\\
       f|_{t=0}=f_{0}.
    \end{cases}
    \end{equation}
\quad Then $f\in \mathcal{C}([0,T];H^{s})$. More precisely, there exists a constant C depending only on $s,p$ and $d$, and such that the following statements hold:

(i)If $s\neq1+\frac{d}{2}$,
$$\|f\|_{H^s}\leq \|f_{0}\|_{H^s}+\int_{0}^{t}\|F(\tau)\|_{H^s}d\tau+c\int_{0}^{t}V^{'}\|F(\tau)\|_{H^s}d\tau,$$
or hence
$$\|f\|_{H^s}\leq e^{CV(t)}(\|f_{0}\|_{H^s}+\int_{0}^{t}e^{-CV(t)}\|F(\tau)\|_{H^s}d\tau),$$
with $V(t)=\int_{0}^{t}\|\nabla\upsilon(\tau)\|_{H^{\frac{d}{2}}\bigcap L^{\infty}}d\tau$ if $s<1+\frac{d}{2}$ and $V(t)=\int_{0}^{t}\|\nabla\upsilon(\tau)\|_{H^{s-1}}d\tau$ else.

(ii))If $f=\upsilon$, then for all $s>0,$ the estimates both above hold with $V(t)=\int_{0}^{t}\|\nabla\upsilon(\tau)\|_{L^{\infty}}d\tau.$
\end{lem}
We also need the following lemma about the boundness of the operator $(1-\partial_x^2)^{-1}$.
\begin{lem}[see \cite{Chemin2004}]\label{2-5}
Let $m\in\R$ and $f$ be an $S^{m}$-multiplier(that is, $f:\R^{d}\rightarrow \R$ is smooth and satisfies that for all multi-index $\alpha,$ there exists a constant $C_{\alpha}$ such that $\forall\xi\in\R^{d}, |\partial^{\alpha}f(\xi)|\leq C_{\alpha}(1+|\xi|)^{m-|\alpha|}).$ Then for all $s\in\R$ and $1\leq p,r\leq\infty,$ the operator $f(D)$ is continuous from $H^{s}$ to  $H^{s-m}$, that is $\|f(D)(u)\|_{H^{s-m}}\leq C\|(u)\|_{H^{s}}$.
\end{lem}
\renewcommand{\theequation}{\thesection.\arabic{equation}}
\setcounter{equation}{0}
\section{Local well-posedness}

This section is devoted to the proof of the local well-posedness of the system \eqref{1.4}:

\begin{proof}[Proof of Theorem \ref{thm-main-1}]
The proof is based on the energy method. We divide it into four steps.
{\bf Step 1: Construction of smooth approximate solution}

We introduce the following approximate system of \eqref{1.4}
\begin{equation}\label{CH-7}
    \begin{cases}
       \partial_{t}\eta^{\varepsilon}-\mathcal{J_\varepsilon}\partial_{x}\eta^{\varepsilon}-\frac{7}{2}\mathcal{J_\varepsilon}\eta^{\varepsilon}\mathcal{J_\varepsilon}\partial_{x}\eta^{\varepsilon} \!\!\!\!
       &=\mathcal{J_\varepsilon}(1-\partial_{x}^{2})^{-1}\partial_{x}[-2\mathcal{J_\varepsilon}\eta^{\varepsilon}-\frac{5}{2}\mathcal{J_\varepsilon}(\eta^{\varepsilon})^{2}\\
       &\quad+\frac{7}{4}\mathcal{J_\varepsilon}(\partial_{x}\eta^{\varepsilon})^{2}+\frac{1}{8}\mathcal{J_\varepsilon}(\eta^{\varepsilon})^{3}-\frac{3}{64}\mathcal{J_\varepsilon}(\eta^{\varepsilon})^{4}],\\
       \eta^{\varepsilon}|_{t=0}=\eta^{\varepsilon}_{0},
    \end{cases}
    \end{equation}
where $\mathcal{J_\varepsilon}$ denotes mollifier operator. The regularized equation \eqref{CH-7} reduces to an ordinary differential system:
\begin{equation}\label{CH-8}
    \begin{cases}
       \partial_{t}\eta^{\varepsilon}=\partial_{x}\eta^{\varepsilon}+\frac{7}{2}\mathcal{J_\varepsilon}[\mathcal{J_\varepsilon}\eta^{\varepsilon}\mathcal{J_\varepsilon}\partial_{x}\eta^{\varepsilon}]+
       \mathcal{J_\varepsilon}(1-\partial_{x}^{2})^{-1}\partial_{x}[-2\mathcal{J_\varepsilon}\eta^{\varepsilon}-\frac{5}{2}\mathcal{J_\varepsilon}(\eta^{\varepsilon})^{2}\\
       \quad\quad+\frac{7}{4}\mathcal{J_\varepsilon}(\partial_{x}\eta^{\varepsilon})^{2}+\frac{1}{8}\mathcal{J_\varepsilon}(\eta^{\varepsilon})^{3}-\frac{3}{64}\mathcal{J_\varepsilon}(\eta^{\varepsilon})^{4}],\\
       \eta^{\varepsilon}|_{t=0}=\eta^{\varepsilon}_{0}.
    \end{cases}
    \end{equation}

The classical Picard Theorem ensures that the \eqref{CH-7} has a unique smooth solution
$\eta^{\varepsilon} \in \mathcal{C}([0,T_{\varepsilon}]; H^s(\mathbb{R}))$
for some $ T_{\varepsilon}>0$.

{\bf Step 2: Uniform estimates to the approximate solutions}

Applying the operator $\Lambda^s$ to the system \eqref{CH-7} and then taking the $L^2$ inner product, we get
\begin{equation}\label{CH-9}
  \begin{split}
\frac{1}{2}\frac{d}{dt}\|\Lambda^s \eta^{\varepsilon}\|_{L^2}^2
&=\int_{\R} \Lambda^s \mathcal{J_\varepsilon}\partial_{x}\eta^{\varepsilon}\cdot \Lambda^s \mathcal{J_\varepsilon}\eta^{\varepsilon}dx+\int_{\R} \Lambda^s [\mathcal{J_\varepsilon}\partial_{x}\eta^{\varepsilon}(\frac{7}{2}\mathcal{J_\varepsilon}\eta^{\varepsilon})]\cdot \Lambda^s \mathcal{J_\varepsilon}\eta^{\varepsilon}dx\\
&\quad+\int_{\R} \mathcal{J_\varepsilon}\Lambda^{s-2} \partial_{x}[-2\mathcal{J_\varepsilon}\eta^{\varepsilon}-\frac{5}{2}\mathcal{J_\varepsilon}(\eta^{\varepsilon})^{2}
+\frac{7}{4}\mathcal{J_\varepsilon}(\partial_{x}\eta^{\varepsilon})^{2}\\
&\quad+\frac{1}{8}\mathcal{J_\varepsilon}(\eta^{\varepsilon})^{3}-\frac{3}{64}\mathcal{J_\varepsilon}(\eta^{\varepsilon})^{4}]\cdot\Lambda^s \mathcal{J_\varepsilon}\eta^{\varepsilon}dx=:I_1+I_2+I_3,
 \end{split}
\end{equation}
and
$$I_1=\int_{\R} \Lambda^s \mathcal{J_\varepsilon}\partial_{x}\eta^{\varepsilon}\cdot \Lambda^s \mathcal{J_\varepsilon}\eta^{\varepsilon} dx=\frac{1}{2}\int_{\R}\partial_{x}(\Lambda^s\mathcal{J_\varepsilon}\eta^{\varepsilon})^{2}dx=0,$$
we may get from a standard commutator's process that
\begin{equation*}\label{CH-10}
  \begin{split}
 I_2&=\int_{\R}(\frac{7}{2}\mathcal{J_\varepsilon}\eta^{\varepsilon})\partial_{x}\Lambda^s \mathcal{J_\varepsilon}\eta^{\varepsilon}\cdot\Lambda^s \mathcal{J_\varepsilon}\eta^{\varepsilon}dx+\int_{\R}[\Lambda^s,(\frac{7}{2}\mathcal{J_\varepsilon}\eta^{\varepsilon})]\mathcal{J_\varepsilon}\partial_{x}\eta^{\varepsilon}\cdot\Lambda^s\mathcal{J_\varepsilon}\eta^{\varepsilon}dx\\
&=:I_{2.1}+I_{2.2}.
 \end{split}
\end{equation*}
For $I_{2.1}$ we get by integration by parts that
\begin{equation}\label{CH-11}
  \begin{split}
 |I_{2.1}|&\leq \frac{1}{2}\|\partial_{x}(\frac{7}{2}\mathcal{J_\varepsilon}\eta^{\varepsilon})\|_{L^{\infty}}\|\Lambda^s \mathcal{J_\varepsilon}\eta^{\varepsilon}\|_{L^2}^2\leq C\|\partial_{x}\mathcal{J_\varepsilon}\eta^{\varepsilon}\|_{L^{\infty}}\|\Lambda^s\mathcal{J_\varepsilon}\eta^{\varepsilon}\|_{L^2}^2\\
&\leq C\|\partial_{x}\mathcal{J_\varepsilon}\eta^{\varepsilon}\|_{L^{\infty}}\|\mathcal{J_\varepsilon}\eta^{\varepsilon}\|^{2}_{H^{s}},
 \end{split}
 \end{equation}
thanks to H\"{o}lder's inequality and commutator estimate, we infer that
\begin{equation}\label{CH-12}
  \begin{split}
|I_{2.2}|&\leq \|[\Lambda^s,(\frac{7}{2}\mathcal{J_\varepsilon}\eta^{\varepsilon})]\mathcal{J_\varepsilon}\partial_{x}\eta^{\varepsilon}\|_{L^2}\|\Lambda^s \mathcal{J_\varepsilon}\eta^{\varepsilon}\|_{L^2}\\
& \leq C\|\mathcal{J_\varepsilon}\eta^{\varepsilon}\|_{H^{s}}(\|\mathcal{J_\varepsilon}\eta^{\varepsilon}\|_{H^{s}}\|\mathcal{J_\varepsilon}\partial_{x}\eta^{\varepsilon}\|_{L^{\infty}}+\|\mathcal{J_\varepsilon}\partial_{x}\eta^{\varepsilon}\|_{L^{\infty}}\|\mathcal{J_\varepsilon}\partial_{x}\eta^{\varepsilon}\|_{H^{s-1}})\\
& \leq C\|\mathcal{J_\varepsilon}\eta^{\varepsilon}\|_{H^{s}}(\|\mathcal{J_\varepsilon}\eta^{\varepsilon}\|_{H^{s}}\|\mathcal{J_\varepsilon}\partial_{x}\eta^{\varepsilon}\|_{L^{\infty}}+\|\mathcal{J_\varepsilon}\partial_{x}\eta^{\varepsilon}\|_{L^{\infty}}\|\mathcal{J_\varepsilon}\eta^{\varepsilon}\|_{H^{s}})\\
& \leq C\|\partial_{x}\mathcal{J_\varepsilon}\eta^{\varepsilon}\|_{L^{\infty}}\|\mathcal{J_\varepsilon}\eta^{\varepsilon}\|^{2}_{H^{s}}.
 \end{split}
 \end{equation}
Substituting \eqref{CH-11} and \eqref{CH-12} into $I_{2}$ leads to
\begin{equation}\label{CH-13}
  \begin{split}
& |I_2|\leq C\|\partial_{x}\mathcal{J_\varepsilon}\eta^{\varepsilon}\|_{L^{\infty}}\|\mathcal{J_\varepsilon}\eta^{\varepsilon}\|^{2}_{H^{s}}.
 \end{split}
\end{equation}

Thanks to the Sobolev embedding theorem $H^{s}\hookrightarrow L^{\infty}(for s>\frac{1}{2})$, we know that
 $$\|\partial_{x}\mathcal{J_\varepsilon}\eta^{\varepsilon}\|_{L^{\infty}}\leq C\|\partial_{x}\mathcal{J_\varepsilon}\eta^{\varepsilon}\|_{H^{s-1}}\leq C\|\mathcal{J_\varepsilon}\eta^{\varepsilon}\|_{H^{s}},$$
where $s>\frac{3}{2}$.

Which along with \eqref{CH-13} implies that
\begin{equation}\label{CH-14}
  \begin{split}
& |I_2|\leq C\|\mathcal{J_\varepsilon}\eta^{\varepsilon}\|^{3}_{H^{s}},
 \end{split}
\end{equation}
because $H^{s-1}(s-1>\frac{1}{2})$ is a Banach algebra, we get from the Sobolev embedding inequality that
\begin{equation}
  \begin{split}
|I_3| &\leq \|\mathcal{J_\varepsilon}\Lambda^{s-2} \partial_{x}[-2\mathcal{J_\varepsilon}\eta^{\varepsilon}-\frac{5}{2}\mathcal{J_\varepsilon}(\eta^{\varepsilon})^{2}
+\frac{7}{4}\mathcal{J_\varepsilon}(\partial_{x}\eta^{\varepsilon})^{2}+\frac{1}{8}\mathcal{J_\varepsilon}(\eta^{\varepsilon})^{3}-\frac{3}{64}\mathcal{J_\varepsilon}(\eta^{\varepsilon})^{4}]\|_{L^2}\|\Lambda^s \mathcal{J_\varepsilon}\eta^{\varepsilon}\|_{L^2}\\
&\leq C\|\partial_{x}[-2\mathcal{J_\varepsilon}\eta^{\varepsilon}-\frac{5}{2}\mathcal{J_\varepsilon}(\eta^{\varepsilon})^{2}
+\frac{7}{4}\mathcal{J_\varepsilon}(\partial_{x}\eta^{\varepsilon})^{2}+\frac{1}{8}\mathcal{J_\varepsilon}(\eta^{\varepsilon})^{3}-\frac{3}{64}\mathcal{J_\varepsilon}(\eta^{\varepsilon})^{4}]\|_{H^{s-2}}\|\mathcal{J_\varepsilon}\eta^{\varepsilon}\|_{H^{s}}\\
&\leq C\|-2\mathcal{J_\varepsilon}\eta^{\varepsilon}-\frac{5}{2}\mathcal{J_\varepsilon}(\eta^{\varepsilon})^{2}
+\frac{7}{4}\mathcal{J_\varepsilon}(\partial_{x}\eta^{\varepsilon})^{2}+\frac{1}{8}\mathcal{J_\varepsilon}(\eta^{\varepsilon})^{3}-\frac{3}{64}\mathcal{J_\varepsilon}(\eta^{\varepsilon})^{4}\|_{H^{s-1}}\|\mathcal{J_\varepsilon}\eta^{\varepsilon}\|_{H^{s}},
 \end{split}
\end{equation}
hence,
\begin{equation}\label{CH-15}
  \begin{split}
|I_3| &\leq C[\|\mathcal{J_\varepsilon}\eta^{\varepsilon}\|_{H^{s-1}}+\|\mathcal{J_\varepsilon}(\eta^{\varepsilon})^{2}\|_{H^{s-1}}+\|\mathcal{J_\varepsilon}(\partial_{x}\eta^{\varepsilon})^{2}\|_{H^{s-1}}\\
&\quad+\|\mathcal{J_\varepsilon}(\eta^{\varepsilon})^{3}\|_{H^{s-1}}+\|\mathcal{J_\varepsilon}(\eta^{\varepsilon})^{4}\|_{H^{s-1}}]\|\mathcal{J_\varepsilon}\eta^{\varepsilon}\|_{H^{s}}\\
& \leq C\|\mathcal{J_\varepsilon}\eta^{\varepsilon}\|^{2}_{H^{s}}(1+2\mathcal{J_\varepsilon}\|\eta^{\varepsilon}\|_{H^{s}}+\|\mathcal{J_\varepsilon}\eta^{\varepsilon}\|^{2}_{H^{s}}+\|\mathcal{J_\varepsilon}\eta^{\varepsilon}\|^{3}_{H^{s}}).
 \end{split}
\end{equation}
Substituting \eqref{CH-14} and \eqref{CH-15} into \eqref{CH-9} leads to
\begin{equation}\label{CH-16}
  \begin{split}
\frac{d}{dt}\|\eta^{\varepsilon}\|^{2}_{H^{s}}&\leq C\|\mathcal{J_\varepsilon}\eta^{\varepsilon}\|^{2}_{H^{s}}(1+2\|\mathcal{J_\varepsilon}\eta^{\varepsilon}\|_{H^{s}}+\|\mathcal{J_\varepsilon}\eta^{\varepsilon}\|^{2}_{H^{s}}+\|\mathcal{J_\varepsilon}\eta^{\varepsilon}\|^{3}_{H^{s}})+C\|\mathcal{J_\varepsilon}\eta^{\varepsilon}\|^{3}_{H^{s}}\\
& \leq C\|\mathcal{J_\varepsilon}\eta^{\varepsilon}\|^{2}_{H^{s}}(1+3\|\mathcal{J_\varepsilon}\eta^{\varepsilon}\|_{H^{s}}+\|\mathcal{J_\varepsilon}\eta^{\varepsilon}\|^{2}_{H^{s}}+\|\mathcal{J_\varepsilon}\eta^{\varepsilon}\|^{3}_{H^{s}}).
 \end{split}
\end{equation}

Therefore, by the bootstrap argument, we may get that, there is a positive time $T$ ($\leq T_{\varepsilon}$) independent of $\varepsilon$ such that for all $\varepsilon>0$,
\begin{equation}\label{CH-17}
  \begin{split}
\sup_{0\leq t\leq T}\|\eta^{\varepsilon}(x,t)\|_{H^s}\leq C\|\eta_{0}\|_{H^s},
 \end{split}
\end{equation}
which along with \eqref{CH-16} implies that
\begin{equation}\label{CH-18}
\{\eta^{\varepsilon}(x,t)\}_{n\in N} \quad \mbox{is uniformly bounded in}  \quad {\mathcal{C}}([0, T]; H^s(\mathbb{R})).
\end{equation}
Furthermore, there holds
\begin{equation}\label{CH-19}
\{\partial_{t}\eta^{\varepsilon}(x,t)\}_{n\in N} \quad \mbox{is uniformly bounded in}  \quad {\mathcal{C}}([0, T]; H^{s-1}(\mathbb{R})).
\end{equation}

{\bf Step 3: Convergence}

With \eqref{CH-17},\eqref{CH-18}, and \eqref{CH-19}, the Aubin-Lions's compactness lemma ensures that there exist a subsequence of $\{\eta^{\varepsilon}(x,t)\}_{\varepsilon>0}$ converges to some limit $\eta(x,t)$ on $[0, T]$
which solves \eqref{1.4}, moreover, there holds
\begin{equation}\label{CH-20}
  \eta(x,t) \in {\mathcal{C}}([0,T];H^s(\mathbb{R}))\cap{\mathcal{C}}^{1}([0, T]; H^{s-1}(\mathbb{R})).
\end{equation}

{\bf Step 4: Uniqueness of the solution}

Let $\eta^{1}$ and $\eta^{2}$ be two solutions of \eqref{CH-9} with the same initial data and satisfy \eqref{CH-18}. We
denote $\eta^{1, 2}:=\eta^1-\eta^2$, then $\eta^{1, 2}$ satisfies
\begin{equation*}
\begin{cases}
    \begin{matrix}
   \partial_{t}\eta^{1, 2}  = \qquad & \!\!\!\!\!\!\!\!\!\!\!\!\!\!\!\!\!\!\!\!\!\!\!
 \partial_{x}\eta^{1, 2}+\frac{7}{2}\eta^{1}\partial_{x}\eta^{1, 2}+\frac{7}{2}\eta^{1,2}\partial_{x}\eta^{ 2}+(1-\partial_{x}^{2})^{-1}\partial_{x}\big(-2\eta^{1,2} & \\
&\!\!\! -\frac{5}{2}\eta^{1,2}(\eta^{2}+\eta^{1})
  +\frac{7}{4}\partial_{x}\eta^{1,2}(\partial_{x}\eta^{1} +\partial_{x}\eta^{2})+\frac{1}{8}\eta^{1,2}((\eta^{1})^{2} &
  \\
&  +\eta^{1}\eta^{2}+(\eta^{2})^{2})
 -\frac{3}{64}\eta^{1,2}(\eta^{2}+\eta^{1})((\eta^{1})^{2}+(\eta^{2})^{2})\big)  & \quad
\forall \, t > 0,\, x \in \R,\\
  \eta^{1, 2}|_{t=0}=0   & & \quad \forall \, x \in \R.
 \qquad \; \,
 \end{matrix}
    \end{cases}
    \end{equation*}
Thanks to transport equation theory, we have
\begin{equation}\label{CH-21}
  \begin{split}
&e^{-C\int_{0}^{t}\|\partial_{x}\eta^{1}(\tau)\|_{H^{s-1}}d\tau}\|\eta^{1,2}(t)\|_{H^{s-1}}\\
&\quad\leq\|\eta_{0}^{1,2}\|_{H^{s-1}}+C\int_{0}^{t}e^{-C\int_{0}^{t}\|\partial_{x}\eta^{1}(\tau^{'})\|_{H^{s-1}}d\tau^{'}}
\times(\|\frac{7}{2}\eta^{1,2}\partial_{x}\eta^{2}\|_{H^{s-1}}\\
&\quad+\|(1-\partial_{x}^{2})^{-1}\partial_{x}\{-2\eta^{1,2}-\frac{5}{2}\eta^{1,2}(\eta^{2}+\eta^{1})
+\frac{7}{4}\partial_{x}\eta^{1,2}(\partial_{x}\eta^{1}+\partial_{x}\eta^{2})\\
&\quad+\frac{1}{8}\eta^{1,2}((\eta^{1})^{2}+\eta^{1}\eta^{2}+(\eta^{2})^{2})-\frac{3}{64}\eta^{1,2}(\eta^{2}+\eta^{1})((\eta^{1})^{2}+(\eta^{2})^{2})\}\|_{H^{s-1}})d\tau.
 \end{split}
\end{equation}
For $s>1+\frac{1}{2}, H^{s-1}$ is an algebra, so we know that
\begin{equation}\label{CH-22}
  \begin{split}
\|\frac{7}{2}\eta^{1,2}\partial_{x}\eta^{2}\|_{H^{s-1}}
\leq C\|\eta^{1,2}\|_{H^{s-1}}\|\partial_{x}\eta^{2}\|_{H^{s-1}}
\leq C\|\eta^{1,2}\|_{H^{s-1}}\|\eta^{2}\|_{H^{s}}.
 \end{split}
\end{equation}
On the other hand, from Lemma \ref{lem-1-D Moser-type estimates} and Lemma \ref{2-5}, we get
\begin{equation}\label{CH-23}
  \begin{split}
&\|(1-\partial_{x}^{2})^{-1}\partial_{x}(-2\eta^{1,2}-\frac{5}{2}\eta^{1,2}(\eta^{2}+\eta^{1}))\|_{H^{s-1}}\\
&\leq C
\|\eta^{1,2}\|_{H^{s-1}}(1+\|\eta^{2}\|_{H^{s-1}}+\|\eta^{1}\|_{H^{s-1}})\leq C\|\eta^{1,2}\|_{H^{s-1}}(1+\|\eta^{2}\|_{H^{s}}+\|\eta^{1}\|_{H^{s}}),
 \end{split}
\end{equation}
and
\begin{equation}\label{CH-24}
  \begin{split}
&\|(1-\partial_{x}^{2})^{-1}\partial_{x}(\frac{7}{4}\partial_{x}\eta^{1,2}(\partial_{x}\eta^{1}+\partial_{x}\eta^{2}))\|_{H^{s-1}}\\
&\leq C\|\partial_{x}\eta^{1,2}\|_{H^{s-2}}(\|\partial_{x}\eta^{1}\|_{H^{s-1}}+\|\partial_{x}\eta^{2}\|_{H^{s-1}})\leq C\|\eta^{1,2}\|_{H^{s-1}}(\|\eta^{1}\|_{H^{s}}+\|\eta^{2}\|_{H^{s}}).
 \end{split}
\end{equation}
\quad Similarly, we have
\begin{equation}\label{CH-25}
  \begin{split}
&\|(1-\partial_{x}^{2})^{-1}\partial_{x}(\frac{1}{8}\eta^{1,2}((\eta^{1})^{2}+\eta^{1}\eta^{2}+(\eta^{2})^{2}))\|_{H^{s-1}}\\
&\quad\leq C\|\eta^{1,2}\|_{H^{s-1}}(\|\eta^{1}\|_{H^{s-1}}^{2}+\|\eta^{1}\|_{H^{s-1}}\|\eta^{2}\|_{H^{s-1}}
+\|\eta^{2}\|_{H^{s-1}}^{2}))\\
&\quad\leq C\|\eta^{1,2}\|_{H^{s-1}}(\|\eta^{1}\|_{H^{s}}^{2}+\|\eta^{1}\|_{H^{s}}\|\eta^{2}\|_{H^{s}}
+\|\eta^{2}\|_{H^{s}}^{2}),
 \end{split}
\end{equation}
and
\begin{equation}\label{CH-26}
  \begin{split}
&\|(1-\partial_{x}^{2})^{-1}\partial_{x}(-\frac{3}{64}\eta^{1,2}(\eta^{2}+\eta^{1})((\eta^{1})^{2}+(\eta^{2})^{2}))\|_{H^{s-1}}\\
&\quad\leq C\|\eta^{1,2}\|_{H^{s-1}}(\|\eta^{1}\|_{H^{s-1}}+\|\eta^{2}\|_{H^{s-1}})
(\|\eta^{1}\|_{H^{s-1}}^{2}+\|\eta^{2}\|_{H^{s-1}}^{2})\\
&\quad\leq C\|\eta^{1,2}\|_{H^{s-1}}(\|\eta^{1}\|_{H^{s}}+\|\eta^{2}\|_{H^{s}})
(\|\eta^{1}\|_{H^{s}}^{2}+\|\eta^{2}\|_{H^{s}}^{2}).
 \end{split}
\end{equation}
Similarly, we get
\begin{equation}\label{CH-27}
  \begin{split}
&\|(1-\partial_{x}^{2})^{-1}\partial_{x}\{-2\eta^{1,2}-\frac{5}{2}\eta^{1,2}(\eta^{2}+\eta^{1})-\frac{7}{4}\partial_{x}\eta^{1,2}(\partial_{x}\eta^{1}
+\partial_{x}\eta^{2})\\
&\quad+\frac{1}{8}\eta^{1,2}((\eta^{1})^{2}+\eta^{1}\eta^{2}+(\eta^{2})^{2})-\frac{3}{64}\eta^{1,2}(\eta^{2}+\eta^{1})((\eta^{1})^{2}+(\eta^{2})^{2})\}\|_{H^{s-1}}\\
&\leq C\|\eta^{1,2}\|_{H^{s-1}}(1+\|\eta^{1}\|_{H^{s}}+\|\eta^{2}\|_{H^{s}}+\|\eta^{1}\|_{H^{s}}^{2}+\|\eta^{2}\|_{H^{s}}^{2}+\|\eta^{1}\|_{H^{s}}\|\eta^{2}\|_{H^{s}}\\
&\quad+(\|\eta^{1}\|_{H^{s}}+\|\eta^{2}\|_{H^{s}})
(\|\eta^{1}\|_{H^{s}}^{2}+\|\eta^{2}\|_{H^{s}}^{2}))\\
&\leq C\|\eta^{1,2}\|_{H^{s-1}}(1+\|\eta^{1}\|_{H^{s}}^{3}+\|\eta^{2}\|_{H^{s}}^{3}).
 \end{split}
\end{equation}
Therefore, from \eqref{CH-22} to \eqref{CH-27},  with Young's inequality,  we obtain
\begin{equation}\label{CH-28}
  \begin{split}
&e^{-C\int_{0}^{t}\|\partial_{x}\eta^{1}(\tau)\|_{H^{s-1}}d\tau}\|\eta^{1,2}(t)\|_{H^{s-1}}\\
&\leq \|\eta_{0}^{1,2}\|_{H^{s-1}}+C\int_{0}^{t}e^{-C\int_{0}^{t}\|\partial_{x}\eta^{1}(\tau^{'})\|_{H^{s-1}}d\tau^{'}}\|\eta^{1,2}(t)\|_{H^{s-1}}(1+\|\eta^{1}\|_{H^{s}}^{3}+\|\eta^{2}\|_{H^{s}}^{3})d\tau.
 \end{split}
\end{equation}
Hence, applying the Gronwall's inequality, we reach
\begin{equation}\label{CH-29}
  \begin{split}
\|\eta^{1}(t)-\eta^{2}(t)\|_{H^{s-1}}\leq\|\eta_{0}^{1,2}\|_{H^{s-1}} e^{-C\int_{0}^{t}1+\|\eta^{1}\|_{H^{s}}^{3}+\|\eta^{2}\|_{H^{s}}^{3}d\tau}.
 \end{split}
\end{equation}
With $\|\eta_{0}^{1,2}\|_{H^{s-1}}=0$, we get that $\|\eta^{1}(t)-\eta^{2}(t)\|_{H^{s-1}}\equiv0$, which implies that $\eta^{1}\equiv\eta^{2},\forall x\in \R, \,t \in[0,T].$

Based on the argument of the proof of uniqueness, we may readily get the map $\eta_{0}\mapsto \eta$ is continuous from a neighborhood of $\eta_{0}$ in $H^{s}$ into $\eta(x,t) \in {\mathcal{C}}([0,T];H^s)\bigcap{\mathcal{C}}^{1}([0, T]; H^{s-1}).$

Therefore, from Step 1 to Step 4, we complete the proof of Theorem \ref{thm-main-1}.
\end{proof}

\renewcommand{\theequation}{\thesection.\arabic{equation}}
\setcounter{equation}{0}

\section{Weave-breaking criteria}

In this section, attention is turned to investigating conditions of wave breaking.

With Theorem \ref{thm-main-1} in hand, we are now ready to complete the proof of the wave-breaking.
The proof of Theorem \ref{thm-main-2} strongly depends on Lemma \ref{lem-Transport equation theory} on the localization analysis for the transport equation.

\begin{proof}[Proof of Theorem \ref{thm-main-2}]

 Applying the operator $\Lambda^s$ to the first equation of system \eqref{1.4} and then taking the $L^2$ inner product, we get
\begin{equation}\label{CH-30}
  \begin{split}
&\frac{1}{2}\frac{d}{dt}\|\Lambda^s \eta\|_{L^2}^2
=\int_{\R} \Lambda^s \partial_{x}\eta\cdot \Lambda^s \eta dx+\int_{\R} \Lambda^s \partial_{x}\eta(\frac{7}{2}\eta)\cdot \Lambda^s \eta dx\\
&\quad+\int_{\R} \Lambda^{s-2} \partial_{x}[-2\eta-\frac{5}{2}\eta^{2}
+\frac{7}{4}(\partial_{x}\eta)^{2}+\frac{1}{8}\eta^{3}-\frac{3}{64}\eta^{4}]\cdot\Lambda^s \eta dx=:I_1+I_2+I_3,
 \end{split}
\end{equation}
and
$$I_1=\int_{\R} \Lambda^s \partial_{x}\eta\cdot \Lambda^s \eta dx=\frac{1}{2}\int_{\R}\partial_{x}(\Lambda^s\eta)^{2}dx=0,$$
we may get from a standard commutator's process that
\begin{equation*}\label{CH-31}
  \begin{split}
I_2&=\int_{\R}(\frac{7}{2}\eta)\partial_{x}\Lambda^s \eta\cdot\Lambda^s \eta dx+\int_{\R}[\Lambda^s,(\frac{7}{2}\eta)]\partial_{x}\eta\cdot\Lambda^s\eta dx=:I_{2.1}+I_{2.2}.
 \end{split}
\end{equation*}
For $I_{2.1}$ we get by integration by parts that
\begin{equation}\label{CH-32}
  \begin{split}
  |I_{2.1}|&\leq \frac{1}{2}\|\partial_{x}(\frac{7}{2}\eta)\|_{L^{\infty}}\|\Lambda^s \eta\|_{L^2}^2\leq C\|\partial_{x}\eta\|_{L^{\infty}}\|\Lambda^s\eta\|_{L^2}^2\leq C\|\partial_{x}\eta\|_{L^{\infty}}\|\eta\|^{2}_{H^{s}},
 \end{split}
 \end{equation}
thanks to H\"{o}lder's inequality and commutator estimate, we infer that
\begin{equation}\label{CH-33}
  \begin{split}
|I_{2.2}|&\leq \|[\Lambda^s,(\frac{7}{2}\eta)]\partial_{x}\eta\|_{L^2}\|\Lambda^s \eta\|_{L^2} \leq C\|\eta\|_{H^{s}}(\|\eta\|_{H^{s}}\|\partial_{x}\eta\|_{L^{\infty}}+\|\partial_{x}\eta\|_{L^{\infty}}\|\partial_{x}\eta\|_{H^{s-1}})\\
& \leq C\|\eta\|_{H^{s}}(\|\eta\|_{H^{s}}\|\partial_{x}\eta\|_{L^{\infty}}+\|\partial_{x}\eta\|_{L^{\infty}}\|\eta\|_{H^{s}}) \leq C\|\partial_{x}\eta\|_{L^{\infty}}\|\eta\|^{2}_{H^{s}}.
 \end{split}
 \end{equation}
Substituting \eqref{CH-32} and \eqref{CH-33} into $I_{2}$ leads to
\begin{equation}\label{CH-34}
  \begin{split}
& |I_2|\leq C\|\partial_{x}\eta\|_{L^{\infty}}\|\eta\|^{2}_{H^{s}}.
 \end{split}
\end{equation}
Because $H^{s-1}(s-1>\frac{1}{2})$ is a Banach algebra, we get from the Sobolev embedding inequality that
\begin{equation}\label{CH-35}
  \begin{split}
\quad\quad\quad|I_3|& \leq \|\Lambda^{s-2} \partial_{x}[-2\eta-\frac{5}{2}\eta^{2}
+\frac{7}{4}(\partial_{x}\eta)^{2}+\frac{1}{8}\eta^{3}-\frac{3}{64}\eta^{4}]\|_{L^2}\|\Lambda^s \eta\|_{L^2}\\
& \leq C\|\partial_{x}[-2\eta-\frac{5}{2}\eta^{2}
+\frac{7}{4}(\partial_{x}\eta)^{2}+\frac{1}{8}\eta^{3}-\frac{3}{64}\eta^{4}]\|_{H^{s-2}}\|\eta\|_{H^{s}}\\
& \leq C\|-2\eta-\frac{5}{2}\eta^{2}
+\frac{7}{4}(\partial_{x}\eta)^{2}+\frac{1}{8}\eta^{3}-\frac{3}{64}\eta^{4}\|_{H^{s-1}}\|\eta\|_{H^{s}}.
 \end{split}
\end{equation}

Hence,
\begin{equation}\label{CH-35}
  \begin{split}
|I_3|& \leq C[\|\eta\|_{H^{s-1}}+\|\eta^{2}\|_{H^{s-1}}+\|(\partial_{x}\eta)^{2}\|_{H^{s-1}}+\|\eta^{3}\|_{H^{s-1}}+\|\eta^{4}\|_{H^{s-1}}]\|\eta\|_{H^{s}}\\
& \leq C[\|\eta\|^{2}_{H^{s}}+\|\eta\|_{L^{\infty}}\|\eta\|^{2}_{H^{s}}+\|\partial_{x}\eta\|_{L^{\infty}}\|\eta\|^{2}_{H^{s}}+\|\eta\|^{2}_{L^{\infty}}\|\eta\|^{2}_{H^{s}}+\|\eta\|^{3}_{L^{\infty}}\|\eta\|^{2}_{H^{s}}]\\
& \leq C\|\eta\|^{2}_{H^{s}}(1+\|\eta\|_{L^{\infty}}+\|\partial_{x}\eta\|_{L^{\infty}}+\|\eta\|^{2}_{L^{\infty}}+\|\eta\|^{3}_{L^{\infty}}),
 \end{split}
\end{equation}
substituting \eqref{CH-34} and \eqref{CH-35} into \eqref{CH-30} leads to
\begin{equation}\label{CH-36}
  \begin{split}
&\frac{d}{dt}\|\eta\|^{2}_{H^{s}}\leq C\|\eta\|^{2}_{H^{s}}(1+\|\eta\|_{L^{\infty}}+\|\partial_{x}\eta\|_{L^{\infty}}+\|\eta\|^{2}_{L^{\infty}}+\|\eta\|^{3}_{L^{\infty}}).
 \end{split}
\end{equation}
\quad Thanks to Gronwall's inequality, one can see
$$\|\eta(t)\|^{2}_{H^{s}}\leq\|\eta_{0}\|^{2}_{H^{s}}e^{C\int_{0}^{t}(1+\|\eta\|_{L^{\infty}}+\|\partial_{x}\eta\|_{L^{\infty}}+\|\eta\|^{2}_{L^{\infty}}+\|\eta\|^{3}_{L^{\infty}})d\tau},$$
using the Sobolev embedding theorem $H^{s}\hookrightarrow L^{\infty}$(for $s>\frac{1}{2})$, we get from Theorem \ref{thm-main-1} that
$$\|\eta(t)\|_{L^{\infty}}\leq C\|\eta_{0}\|_{H^{1}},$$
which implies that
$$\|\eta(t)\|^{2}_{H^{s}}\leq\|\eta_{0}\|^{2}_{H^{s}}e^{C\int_{0}^{t}(1+C_{1}+\|\partial_{x}\eta\|_{L^{\infty}}d\tau)},$$
where $C_{1}=C_{1}(\|\eta_{0}\|_{H^{1}}).$

Therefore, if the maximal existence time
$T_{\eta_{0}}^{\ast}<\infty$ satisfies $\int_{0}^{T_{\eta_{0}}^{\ast}}\|\partial_{x}\eta(\tau)\|_{L^{\infty}}d\tau<\infty,$
then it implies that
$$\limsup_{t\rightarrow T_{\eta_{0}}^{\ast}}(\|\eta(t)\|_{H^{s}})<\infty,$$
contradicts the assumption on the maximal existence time $T_{\eta_{0}}^{\ast}<\infty$. Hence, the proof of Theorem \ref{thm-main-2} is complete.
\end{proof}
\begin{lem}\label{lem-4.2}
Let $\eta_0\in H^{s}(\R)$ with $s>\frac{3}{2},$ and let $T>0$ be the maximal existence time of the solution $\eta$ to the \eqref{1.4} with initial data $\eta_{0}.$ Then the corresponding solution blows up in finite time if and only if
\begin{equation}\label{CH-p}
  \begin{split}
\lim_{t\rightarrow T^{-}}\sup_{x\in\R}\eta_{x}(t,x)=+\infty.
 \end{split}
\end{equation}
\end{lem}
\begin{proof}[Proof of Lemma \ref{lem-4.2}]
Applying a simple density argument, we only need to show that Lemma \ref{lem-4.2} holds with some $s\geq3.$ Here we assume $s=3$ to prove the above Lemma.\\

Multiplying the equation \eqref{CH-3} by $\eta$ and integrating by parts, we get
\begin{equation}\label{CH-a}
  \begin{split}
\frac{d}{dt}\int_{\R}(\eta^{2}+\eta_{x}^{2})dx=0.
 \end{split}
\end{equation}

On the other hand, multiplying equation \eqref{CH-3} by $\eta_{xx}$ and integrating by parts, we get
\begin{equation}\label{CH-b}
  \begin{split}
\frac{d}{dt}\int_{\R}(\eta_{x}^{2}+\eta_{xx}^{2})dx=3\int_{\R}\big(\eta\eta_{x}\eta_{xx}(1-\frac{1}{4}\eta
+\frac{1}{8}\eta^{2})+\frac{7}{2}\eta_{x}\eta_{xx}^{2}\big) dx.
 \end{split}
\end{equation}

Therefore, combining \eqref{CH-a} with \eqref{CH-b}, one can see that
\begin{equation}\label{CH-c}
  \begin{split}
\frac{d}{dt}\int_{\R}(\eta^{2}+2\eta_{x}^{2}+\eta_{xx}^{2})dx=3\int_{\R}\big(\eta\eta_{x}\eta_{xx}(1-\frac{1}{4}\eta
+\frac{1}{8}\eta^{2})+\frac{7}{2}\eta_{x}\eta_{xx}^{2}\big) dx.
 \end{split}
\end{equation}

From interpolation inequality we know that
\begin{equation}\label{CH-b1}
  \begin{split}
3\int_{\R}\eta\eta_{x}\eta_{xx}dx&=-\frac{3}{2}\int_{\R}\eta_{x}^{3}dx\leq\frac{3}{2}\|\eta_{x}\|^{3}_{L^{3}(\R)}
\leq\frac{3}{2}C\|\eta_{x}\|^{\frac{5}{2}}_{L^{2}(\R)}\|\eta_{xx}\|^{\frac{1}{2}}_{L^{2}(\R)}\\
&\leq\|\eta_{xx}\|^{2}_{L^{2}(\R)}+C\|\eta_{x}\|^{\frac{10}{3}}_{L^{2}(\R)},
 \end{split}
\end{equation}
\begin{equation}\label{CH-b2}
  \begin{split}
-\frac{3}{4}\int_{\R}\eta^{2}\eta_{x}\eta_{xx} dx&=\frac{3}{4}\int_{\R}\eta\eta_{x}^{3} dx\leq\frac{3}{4}\|\eta\|_{L^{\infty}(\R)}\|\eta_{x}\|^{3}_{L^{3}(\R)}\leq\frac{3}{8}\|\eta\|^{2}_{H^{1}(\R)}+\frac{3}{8}\|\eta_{x}\|^{6}_{L^{3}(\R)}\\
&\leq\frac{3}{8}\|\eta\|^{2}_{H^{1}(\R)}+\frac{3}{8}C\|\eta_{x}\|^{5}_{L^{2}(\R)}\|\eta_{xx}\|_{L^{2}(\R)}\\
&\leq\frac{3}{8}\|\eta_{0}\|^{2}_{H^{1}(\R)}
+\frac{3}{32}C\|\eta_{x}\|^{10}_{L^{2}(\R)}+\|\eta_{xx}\|^{2}_{L^{2}(\R)},
 \end{split}
\end{equation}
\begin{equation}\label{CH-b3}
  \begin{split}
\frac{3}{8}\int_{\R}\eta^{3}\eta_{x}\eta_{xx} dx&=-\frac{9}{16}\int_{\R}\eta^{2}\eta_{x}^{3} dx\leq\frac{9}{32}\|\eta_{0}\|^{4}_{H^{1}(\R)}+\frac{9}{32}\|\eta_{x}\|^{6}_{L^{3}(\R)}\\
&\leq\frac{9}{32}\|\eta_{0}\|^{4}_{H^{1}(\R)}+\|\eta_{xx}\|^{2}_{L^{2}(\R)}+\frac{9}{128}C\|\eta_{x}\|^{10}_{L^{2}(\R)}.
 \end{split}
\end{equation}

Assume that $T<+\infty$ and there exists $M>0$ such that
\begin{equation}\label{CH-d}
  \begin{split}
\eta_{x}(t,x)\leq M, \quad \forall(t,x)\in[0,T)\times\R.
 \end{split}
\end{equation}

Therefore, from \eqref{CH-b1} to \eqref{CH-d}, we obtain
\begin{equation}\label{CH-e}
  \begin{split}
\frac{d}{dt}\int_{\R}(\eta^{2}+2\eta_{x}^{2}+\eta_{xx}^{2})dx&\leq 3\|\eta_{xx}\|_{L^{2}}^{2}+C_{2}+\int_{\R}\frac{21}{2}\eta_{x}\eta_{xx}^{2}dx\\
&\leq(3+\frac{21}{2}M)\|\eta_{xx}\|_{L^{2}}^{2}+C_{2},\\
 \end{split}
\end{equation}
where $C_{2}=C\|\eta_{0}\|_{H^{1}(\R)}^{\frac{10}{3}}+\frac{3}{8}\|\eta_{0}\|_{H^{1}(\R)}^{2}
+C\|\eta_{0}\|_{H^{1}(\R)}^{10}+\frac{9}{32}\|\eta_{0}\|_{H^{1}(\R)}^{4}.$

Applying Gronwall's inequality to \eqref{CH-e} yields for every $t\in[0,T)$
\begin{equation}\label{CH-f}
  \begin{split}
\|\eta(t)\|_{H^{2}(\R)}^{2}\leq (2\|\eta_{0}\|_{H^{2}(\R)}^{2}+C_{2}T)e^{(3+\frac{21}{2}M)T}.
 \end{split}
\end{equation}

Differentiating equation \eqref{CH-3} with respect to $x$, and multiplying the result equation by $\eta_{xxx},$ then integrating by parts, we have
\begin{equation}\label{CH-g}
  \begin{split}
\frac{d}{dt}\int_{\R}(\eta_{xx}^{2}+\eta_{xxx}^{2})dx=&-\frac{15}{2}\int_{\R}\eta_{x}\eta_{xx}^{2}dx+\frac{15}{4}\int_{\R}\eta\eta_{x}\eta_{xx}^{2}dx
-\frac{9}{4}\int_{\R}\eta\eta_{x}^{3}\eta_{xx}dx\\
&-\frac{45}{16}\int_{\R}\eta^{2}\eta_{x}\eta_{xx}^{2}dx+\frac{35}{4}\int_{\R}\eta_{x}\eta_{xxx}^{2}dx,
 \end{split}
\end{equation}
we only need to know
\begin{equation}\label{CH-g1}
  \begin{split}
-\frac{9}{4}\int_{\R}\eta\eta_{x}^{3}\eta_{xx}dx\leq\frac{9}{16}\|\eta_{x}\|_{L^{5}}^{5}\leq\frac{9}{16}C\|\eta_{x}\|_{L^{2}}^{\frac{7}{2}}\|\eta_{xx}\|_{L^{2}}^{\frac{3}{2}}
\leq\|\eta_{xx}\|_{L^{2}}^{2}+C\|\eta_{0}\|_{H^{1}(\R)}^{14},
 \end{split}
\end{equation}
which implies that
\begin{equation}\label{CH-g2}
  \begin{split}
\frac{d}{dt}\int_{\R}(\eta_{xx}^{2}+\eta_{xxx}^{2})dx&\leq\frac{15}{2}M\int_{\R}\eta_{xx}^{2}dx+\frac{15}{4}\|\eta\|_{L^{\infty}}M\int_{\R}\eta_{xx}^{2}dx
+\|\eta_{xx}\|_{L^{2}}^{2}\\
&+C\|\eta_{0}\|_{H^{1}(\R)}^{14}
+\frac{45}{16}\|\eta\|^{2}_{L^{\infty}}M\int_{\R}\eta_{xx}^{2}dx+\frac{35}{4}M\int_{\R}\eta_{xxx}^{2}dx\\
&\leq(1+C_{3}M)\int_{\R}(\eta_{xx}^{2}+\eta_{xxx}^{2})dx+C_{4},
 \end{split}
\end{equation}
where $C_{3}=\frac{35}{4}+\frac{15}{4}\|\eta_{0}\|_{H^{1}(\R)}+\frac{45}{16}\|\eta_{0}\|_{H^{1}(\R)}^{2},$ $C_{4}=C\|\eta_{0}\|_{H^{1}(\R)}^{14}$ and we have used the assumption \eqref{CH-d}. Hence, applying Gronwall's inequality implies that
\begin{equation}\label{CH-h}
  \begin{split}
\int_{\R}(\eta_{xx}^{2}+\eta_{xxx}^{2})dx\leq(\|\eta_{0xx}\|_{H^{1}(\R)}^{2}+C_{4}T)e^{(1+C_{3}M)T},
 \end{split}
\end{equation}
together with \eqref{CH-f} yields for every $t\in[0,T)$
\begin{equation}\label{CH-i}
  \begin{split}
\|\eta(t)\|_{H^{3}(\R)}^{2}\leq[2\|\eta_{0}\|_{H^{3}(\R)}^{2}+(C_{2}+C_{4})T]e^{[4+(\frac{21}{2}+C_{3})M]T},
 \end{split}
\end{equation}
which contradicts the assumption the maximal existence time $T<+\infty$.

Conversely, the Sobolev embedding theorem $H^{s}(\R)\hookrightarrow L^{\infty}(\R)($with $s>\frac{1}{2})$ implies that if \eqref{CH-p} holds, the corresponding solution blows up in finite time, which completes the proof of Lemma \ref{lem-4.2}.
\end{proof}
\begin{lem}[see \cite{Constantin1998}]\label{lem-4.3}
Let $T>0$ and $\eta\in\mathcal{C}^{1}([0,T);H^{2}(\R)).$ Then for $\forall t\in[0,T),$ there exists at least one point $\xi(t)\in\R$ with
$$M(t)=\inf_{x\in \R}(\eta_{x}(t,x))=\eta_{x}(t,\xi(t)).$$
The function $M(t)$ is absolutely continuous on $(0,T)$ with
$$\frac{dM(t)}{dt}=\eta_{tx}(t,\xi(t)) \quad a.e. \quad on \quad(0,T).$$
\end{lem}

\begin{proof}[Proof of Theorem \ref{thm-main-3}]

The technique used here is inspired from \cite{Gui2011}. Similar to the proof of Lemma \ref{lem-4.2}, we assume $s=3$ to prove Theorem \ref{thm-main-3}, now we consider the Lagrangian scale of \eqref{1.4} the initial value problem
\begin{equation*}
\left\{
    \begin{array}{l}
        \displaystyle\frac{\partial q}{\partial t}=\eta(t,q), \ \ \ \ \ \ \ \ \ \ \ \forall 0<t<T,x\in \R,\\
        q(0,x)=x, \ \ \ \ \ \ \ \ \ \ \ \ \ \forall x\in \R,
 \end{array}
    \right.
    \end{equation*}
where $\eta\in \mathcal{C}([0,T);H^{s})$ with $s>\frac{3}{2}$, and $T>0$ being the maximal time of existence. A direct calculation also yields $q_{tx}(t,x)=\eta_{x}(t,q(t,x))q_{x}(t,x)$. Hence for $t>0,x\in \R$, we have
$$q_{x}(t,x)=e^{\int_{0}^{t}\eta_{x}(\tau,q(\tau,x))}d\tau>0,$$
which implies that $q(t,\cdot): \R\rightarrow \R$ is a diffeomorphism of the line for every $t\in[0,T)$. By Lemma \ref{lem-4.3}, we know that $\xi(t)$ [with $t\in[0,T)$] exists such that
$$M(t):=\eta_{x}(t,\xi(t))=\inf_{x\in \R}(\eta_{x}(t,x))\quad \forall t\in[0,T).$$
\quad And then
\begin{equation}\label{CH-37}
\eta_{xx}(t,\xi(t))=0\quad for \quad a.e.\quad t\in[0,T).
\end{equation}
\quad On the other hand, since $q(t,\cdot):\R\rightarrow\R$ is the diffeomorphism for every $t\in[0,T),$ there exists $x_{1}(t)\in \R$ such that
$$q(t,x_{1}(t))=\xi(t)\quad \forall t\in[0,T).$$

Differentiating both sides of the first equation of \eqref{1.4} with respect to $x$, and we get
\begin{equation}\label{CH-38}
  \begin{split}
\eta_{tx}=&\eta_{xx}+\frac{7}{2}\eta_{x}^{2}+\frac{7}{2}\eta\eta_{xx}+(- 2g_{x}\ast\eta_{x})+(-5g_{x}\ast\eta\eta_{x})\\
&+\frac{7}{2}g_{x}\ast\eta_{x}\eta_{xx}+\frac{3}{8}g_{x}\ast\eta^{2}\eta_{x}
+(-\frac{3}{16}g_{x}\ast\eta^{3}\eta_{x}).
 \end{split}
\end{equation}
\quad Given $x\in\R,$ let
$$M(t)=\eta_{x}(t,\xi(t)), \quad \forall t\in[0,T),$$
we have
\begin{equation}\label{CH-39}
\frac{d}{dt}M(t)=\frac{7}{2}M^{2}(t)+f(t,\xi(t)),
\end{equation}
for $t\in[0,T),$ where $f$ represents the function

\begin{equation}\label{CH-40}
  \begin{split}
f(t,\xi(t))=&\big((-2g_{x}\ast\eta_{x})+(-5g_{x}\ast\eta\eta_{x})\\
&+\frac{3}{8}g_{x}\ast\eta^{2}\eta_{x}
+(-\frac{3}{16}g_{x}\ast\eta^{3}\eta_{x})\big)(t,\xi(t)).
 \end{split}
\end{equation}
\quad We use the fact that $H(\eta)$ is the conservation law of the \eqref{1.4}. On the other hand, the continuous embedding of $H^{1}(\R)$ into $L^{\infty}(\R)$, applying Young's inequality and $g(x)=\frac{1}{2}e^{-|x|}$ lead to
\begin{equation}\label{CH-41}
  \begin{split}
-2g_{x}\ast\eta_{x}(t,\xi(t))&\geq-2|g_{x}\ast\eta_{x}|\geq -2\|g_{x}\|_{L^{2}}\|\eta_{x}\|_{L^{2}}\geq-\|\eta_{x}\|_{L^{2}}\\
&\geq-(\frac{1}{2}+\frac{1}{2}\|\eta_{x}\|_{L^{2}}^{2})\geq-(\frac{1}{2}+\frac{1}{2}\|\eta_{0}\|_{H^{1}}^{2}),
 \end{split}
\end{equation}
and
\begin{equation}\label{CH-42}
-5g_{x}\ast\eta\eta_{x}(t,\xi(t))\geq-5|g_{x}\ast\eta\eta_{x}|\geq-5\|g_{x}\|_{L^{\infty}}\|\eta\eta_{x}\|_{L^{1}}\geq-\frac{5}{2}\|\eta_{0}\|_{H^{1}}^{2}.
\end{equation}
\quad Similarly, we have
\begin{equation}\label{CH-43}
  \begin{split}
|\frac{3}{8}g_{x}\ast\eta^{2}\eta_{x}|(t,\xi(t))&\leq\frac{3}{8}\|\eta^{2}\eta_{x}\|_{L^{1}}\leq\frac{3}{16}\|\eta\|_{L^{\infty}}\int_{\R}\eta^{2}+\eta_{x}^{2}dx\\
&\leq\frac{3}{16}\|\eta_{0}\|_{H^{1}}\int_{\R}\eta_{0}^{2}+\eta_{0x}^{2}dx\leq\frac{3}{16}\|\eta_{0}\|_{H^{1}}^{3},
 \end{split}
\end{equation}
so we know that
\begin{equation}\label{CH-44}
\frac{3}{8}g_{x}\ast\eta^{2}\eta_{x}(t,\xi(t))\geq-\frac{3}{16}\|\eta_{0}\|_{H^{1}}^{3},
\end{equation}
and
\begin{equation}\label{CH-45}
-\frac{3}{16}g_{x}\ast\eta^{3}\eta_{x}(t,\xi(t))\geq-\frac{3}{16}|g_{x}\ast\eta^{3}\eta_{x}|\geq-\frac{3}{16}\|\eta^{3}\eta_{x}\|_{L^{1}}
\geq-\frac{3}{32}\|\eta_{0}\|_{H^{1}}^{4}.
\end{equation}

Combining \eqref{CH-41} to \eqref{CH-45}, we get
$$f(t,\xi(t))\geq-(\frac{1}{2}+3\|\eta_{0}\|_{H^{1}}^{2}+\frac{3}{16}\|\eta_{0}\|_{H^{1}}^{3}+\frac{3}{32}\|\eta_{0}\|_{H^{1}}^{4}),$$
then we have
\begin{equation}\label{CH-47}
\frac{d}{dt}M(t)\geq\frac{7}{2}M^{2}(t)-C_{0},
\end{equation}
where $C_{0}=\frac{1}{2}+3\|\eta_{0}\|_{H^{1}}^{2}+\frac{3}{16}\|\eta_{0}\|_{H^{1}}^{3}+\frac{3}{32}\|\eta_{0}\|_{H^{1}}^{4}.$\\

By the assumption $M(0)=\eta_{0x}(x_{0})>\sqrt{\frac{2}{7}C_{0}},$ we have $M^{2}(0)>\frac{2}{7}C_{0}.$ We now claim that is true for any $t\in[0,T).$ In fact, assuming the contrary would, in view of $M(t)$ being continuous, ensure the existence of of $t_{0}\in[0,T)$ such that $M^{2}(t)>\frac{2}{7}C_{0}$ for $t\in[0,t_{0})$ but $M^{2}(t_{0})=\frac{2}{7}C_{0},$
combining this with \eqref{CH-47} would give
$$\frac{d}{dt}M(t)\geq 0\quad a.e.\quad on \quad[0,t_{0}).$$
\quad Since $M(t)$ is absolutely continuous on $[0,t_{0}],$ an integration of this inequality would give the following inequality and we get the contradiction
$$M(t_{0})>M(0)=\eta_{0x}(x_{0})>\sqrt{\frac{2}{7}C_{0}},$$
this proves the previous claim.

Using this together with \eqref{CH-47} and the absolute continuity of the function $M(t),$ we see that $M(t)$ is strictly increasing on $[0,T).$ Therefore, choose that $\sigma\in(0,1)$ such that $\sqrt{\sigma}M(0)=\sqrt{\frac{2}{7}C_{0}},$ then we get from \eqref{CH-47} that
$$\frac{d}{dt}M(t)\geq\frac{7}{2}M^{2}(t)-\frac{7}{2}\sigma M^{2}(0)\geq\frac{7(1-\sigma)}{2}M^{2}(t)\quad a.e.\quad on \quad[0,T).$$
\quad Since $M$ is locally Lipschitz on $[0,T)$ and strictly positive, it follows that $\frac{1}{M}$ is locally Lipschitz on $[0,T).$ This gives
\begin{equation}\label{CH-48}
\frac{d}{dt}\Big(\frac{1}{M(t)}\Big)=-\frac{1}{M^{2}(t)}\frac{d}{dt}M(t)\leq-\frac{7(1-\sigma)}{2}\quad a.e.\quad on \quad[0,T).
\end{equation}
\quad Integration of this inequality yields
$$\frac{1}{M(t)}-\frac{1}{M(0)}\leq-\frac{7(1-\sigma)}{2}t,\quad t\in[0,T).$$
\quad Since $M(t)>0$ on $[0,T),$ we get the maximal existence time $T\leq\frac{2}{7(1-\sigma)M(0)}<\infty.$ Moreover, thanks to $M(0)=\eta_{0x}(x_{0})>0$ again, \eqref{CH-48} implies that $$\eta_{x}(t,\xi(t))=M(t)\geq\frac{\eta_{0x}(x_{0})}{1-\frac{7(1-\sigma)}{2}t\eta_{0x}(x_{0})}\rightarrow+\infty,$$
as $t\rightarrow\frac{2}{7(1-\sigma)\eta_{0x}(x_{0})}$. Therefore, thanks to Lemma \ref{lem-4.2}, we complete the proof of Theorem \ref{thm-main-3}.

\end{proof}

\noindent {\bf Acknowledgments.} The author would like to thank professor Guilong Gui for his valuable comments and suggestions. The authors are partially supported by the National Natural Science Foundation of China under the grants 11571279, 11331005, and 11601423.

 \end{document}